\documentclass{amsart}

\usepackage{amssymb,enumerate}
\usepackage{amsthm}
\usepackage[all]{xy}
\usepackage{hyperref}

\newtheorem{lem}{Lemma}[section]
\newtheorem{cor}[lem]{Corollary}
\newtheorem{Fact}[lem]{Fact}
\newtheorem{prop}[lem]{Proposition}
\newtheorem{thm}[lem]{Theorem}

\newtheorem{Defn}[lem]{Definition}
\newtheorem{Notn}[lem]{Notation}
\newtheorem{Ex}[lem]{Example}
\newtheorem{Question}[lem]{Question}
\newtheorem{Property}[lem]{Property}
\newtheorem{Properties}[lem]{Properties}
\newtheorem{Discussion}[lem]{Remark}
\newtheorem{Construction}[lem]{Construction}
\newtheorem{Subprops}{}[lem]
\newtheorem{Para}[lem]{}

\newenvironment{ex}{\begin{Ex}\rm}{\end{Ex}}

\newenvironment{para}{\begin{Para}\rm}{\end{Para}}
\newenvironment{disc}{\begin{Discussion}\rm}{\end{Discussion}}

\newcommand{\ideal}[1]{\mathfrak{#1}}
\newcommand{\m}{\ideal{m}}

\newcommand{\p}{\ideal{p}}
\newcommand{\q}{\ideal{q}}

\newcommand{\pd}{\operatorname{pd}}

\newcommand{\depth}{\operatorname{depth}}

\newcommand{\amp}{\operatorname{amp}}

\newcommand{\Hom}{\operatorname{Hom}}

\newcommand{\cx}{\operatorname{cx}}

\renewcommand{\geq}{\geqslant}
\renewcommand{\leq}{\leqslant}

\newcommand{\G}{\mathcal{G}}

\newcommand{\Ext}{\mbox{Ext}\,}

\newcommand{\Tor}{\mbox{Tor}\,}
\newcommand{\Supp}{\mbox{Supp}\,}
\newcommand{\Spec}{\mbox{Spec}\,}

\newcommand{\gr}{\operatorname{grade}}

\renewcommand{\dim}{\operatorname{dim}}

\newcommand{\gd}{\mbox{G-dim}\,}
\newcommand{\Ci}{\operatorname{CI\text{-}dim}}
\newcommand{\qpd}{\operatorname{qpd}}
\newcommand{\Rfd}{\operatorname{Rfd}}

\newcommand{\uhom}{{\mathbf R}\Hom}
\newcommand{\utp}{\otimes^{\mathbf L}}

\renewcommand{\H}{\mbox{H}}

\begin{document}

\bibliographystyle{amsplain}

\author{Tirdad Sharif}
\address{School of Mathematics, Institute for Research in Fundamental Sciences (IPM), P.O. Box: 19395-5746, Tehran
Iran.}\email{sharif@ipm.ir}
\thanks{The author was supported by a grant from IPM, (No. 83130311).}

\title[Quasi Projective dimension]{Quasi Projective dimension for complexes}

\subjclass[2000]{13C15,13D05,13D09,13D25}

\keywords{Complete intersection, Depth formula, Homological
dimensions, Intersection Theorem.}

\begin{abstract}
\indent In this note, we extend the quasi-projective
dimension of finite (that is, finitely generated)
modules to homologically finite complexes, and we investigate
some of homological properties of this dimension.
\end{abstract}

\maketitle

\section{Introduction}
Throughout, all rings are commutative and Noetherian.
In~\cite{Av} Avramov, and in \cite{AGP}
Avramov, Gasharov and Peeva defined and studied the complexity, the complete intersection dimension,
and the quasi projective dimension of finite modules.

Let ${\mathcal D}_{b}^{f}(R)$ be the category of \emph{homologically finite} $R$-complexes, and $X\in{\mathcal D}_{b}^{f}(R)$.
The complexity and the complete intersection dimension of $X$,
denoted by $\cx_R X$ and $\Ci_R X$, respectively,
were defined and studied by Sather-Wagstaff \cite{Sa}.
In this work, we introduce the quasi-projective dimension, as a refinement of the projective
dimension, for homologically finite complexes
and verify some of it's homological properties analogous to those holding for modules.

Let $Y, Z\in{\mathcal D}_{b}^{f}(R).$ In our main result, Theorem~\ref{T3}, as an application of
the Intersection Theorem for the quasi-projective
dimension, Proposition~\ref{T1}(b), and the depth formula for
the complete intersection dimension, Proposition~\ref{T2}, a lower bound and an upper bound for
$-\sup\uhom_R(Z,(X\utp_R Y))$ with respect to
$\Ci_R X,$ $\cx_R(X)$ and $-\sup\uhom_R(Z,Y)$ is determined, when $\sup(X\utp_R Y)<\infty$ and $\Ci_R X<\infty.$
This result is as an extension of
a grade inequality in \cite[(3.3)]{SY} to complexes; see Example~\ref{ex1} about this inequality.

\section{Homology theory of complexes}
\label{background}
In this paper, definitions and results are formulated within the framework of the derived
category of complexes. The reader is referred to \cite{AF, Chr1, F, F1} for details of the following brief
summary. Let $X$ be a complex of $R$-modules and $R$-homomorphisms. For an integer $n$,
the $n$-th \emph{shift} or \emph{suspension} of $X$ is the complex $\Sigma^{n}X$ with
$(\Sigma^{n}X)_{\ell}=X_{\ell-n}$ and $\partial_{\ell}^{\Sigma^{n}X}=(-1)^{n}\partial_{\ell-n}^{X}$ for
each $\ell$. The $n$-th \emph{cokernel} of $X$ is
$C_n^{X}$=\emph{cokernel} $\partial^X_{n+1}.$ The \emph{supremum} and the \emph{infimum} of a complex $X$, denoted by $\sup(X)$ and $\inf(X)$, respectively, are defined by the supremum and the infimum of $\{i|\H_i(X)\neq0\}$ and let $\amp(X)=\sup(X)-\inf(X).$
The symbol ${\mathcal D}(R)$ denotes the derived category of
$R$-complexes. The full subcategories
${\mathcal D}_{-} (R)$, ${\mathcal D}_{+} (R)$, ${\mathcal D}_{b}
(R)$ and ${\mathcal D}_{0} (R)$ of ${\mathcal D}(R)$ consist of
$R$-complexes $X$ while $\H_{\ell}(X)=0$, for respectively
$\ell\gg 0$, $\ell\ll 0$, $|\ell|\gg 0$ and $\ell\neq 0$. By
${\mathcal D}^f$ we denote the full subcategory of
complexes with all homology modules are finite, called \emph{homologically degreewise finite} complexes. A
complex $X$ is called \emph{homologically finite}, if it is homologically both bounded and degreewise finite.
The right derived functor of the homomorphism functor of $R$-complexes
and the left derived functor of the tensor product of $R$-complexes are denoted by
$\uhom_R(-,-)$ and $-\utp_{R}-$, respectively. A homology isomorphism is a morphism $\alpha:X\to Y$ such that
$\H(\alpha)$ is an isomorphism; homology isomorphisms are marked by
the sign $\simeq$, while $\cong$ is used for isomorphisms. The
equivalence relation generated by the homology isomorphisms is also
denoted by $\simeq$.

\begin{para}\label{para1}
Let $X$ and $Y$ be in ${\mathcal D}_{+}(R)$, then there is an inequality
$$\inf(X\utp_{R}Y)\geq\inf X+\inf Y.$$

\noindent Equality holds if $i=\inf X$ and $j=\inf Y$ are finite and
${\H_{i}(X)\otimes_{R}\H_{j}(Y)\neq0}.$
\end{para}

\begin{para}
The \emph{support} of a complex $X$,
$\Supp(X)$, consists of all $\p\in\Spec(R)$ such that the
$R_{\p}$-complex $X_{\p}$ is not homologically trivial.
\end{para}

\begin{para}\label{para3}
Let $R$ be a ring. If $(R,\m,k)$ is local, \emph{depth} of a complex
$X\in{\mathcal D}_{-}(R)$ is defined as the following
$$\depth_R X=-\sup\uhom_R(k,X).$$
\end{para}

\noindent (a) Let $X\in{\mathcal D}_{b}^{f}(R)$, then the following inequality holds
$$\depth_R X\leq\depth_{R_{\p}}X_{\p}+\dim R/{\p}.$$

\noindent The \emph{dimension} of a complex $X\in{\mathcal
D}_{+}(R)$ is defined by the following formula
$$\dim_R X=\sup\{\dim
R/{\p}-\inf X_{\p}|\p\in\Supp_R X\}.$$

\vspace{0.05in}\noindent (b) Let $X\in{\mathcal D}_{+}^{f}(R)$, then the
following inequality holds
$$\dim_{R_{\p}}X_{\p}+\dim R/{\p}\leq\dim_RX.$$

\vspace{0.05in}\noindent (c) Let $Y\in{\mathcal D}_{+}^{f}(R)$ and $X\in{\mathcal
D}_{-}(R)$, then the next equality holds
$$-\sup\uhom_R(Y,X)=\inf\{\depth_{R_{\p}}X_{\p}+\inf Y_{\p}|\p\in{\Supp X}\cap{\Supp Y}\}.$$

\begin{para}
For a ring $R$, let $(-)^*=\Hom _R(-, R)$. A finite $R$-module $M$ is \emph{totally reflexive}
over $R$ if $M$ is reflexive and $\Ext^i_R(M, R)=0=\Ext^i_R(M^*,R)$ for all $i>0$. Let $X\in{\mathcal D}_{b}^{f}(R)$.
A \emph{G-resolution} of $X$ is a complex $\G\simeq X$, such that each $\G_i$ is totally reflexive over $R$.
The Gorenstein dimension of $X$ is
$$\gd_R(X)=\inf\{\sup\{i|G_i\neq 0\}|\text{G is a G-resolution of X}\}.$$
\end{para}

\section{\bf Quasi-projective dimension for complexes.}

\noindent In this section, all rings are local. We introduce the quasi-projective dimension, as a refinement of the projective
dimension for homologically finite complexes.

\noindent At first, we bring some notions and
definitions.

\begin{para}
Let $R$ be a ring, a codimension c, quasi-deformation of $R$ is a diagram of local
homomorphisms $R\rightarrow R'\leftarrow Q$ such that the first map is flat and the second map is
surjective with kernel generated by a $Q$-sequence of length c.
The complete intersection dimension and the complexity of a complex $X\in{\mathcal D}_{b}^{f}(R)$ are defined
analogous to those of modules as the following
$$\Ci_R X=\inf\{\pd_Q X'-\pd_Q R'|
R\rightarrow R'\leftarrow Q  \text{ is a quasi-deformation}
\},$$
\noindent where $X'=X\otimes_R R'$, and
$$\cx_R(X)=\inf\{d\in\mathbb{N}_{0}|
\beta_{n}(X)\leq\gamma n^{d-1} \text{ for some }
\gamma\in\mathbb{R}\}$$

\vspace{0.05in}\noindent in which, $\beta_{n}(X)=\dim_{k}\H_n(X\utp_R k)$
is the n-th \emph{Betti number} of $X$.

\vspace{0.05in}\noindent Now we define the quasi-projective dimension of $X\in{\mathcal D}_{b}^{f}(R)$
similar to that of modules as the following
$$\qpd_R X=\inf\{\pd_Q X'| R\rightarrow R'\leftarrow Q \text{ is
a quasi-deformation} \}.$$
\end{para}

\vspace{0.05in}\noindent In the first part of the following proposition, we extend~\cite[(5.11)]{AGP} to complexes.
In the second part, we show that in the Intersection Theorem, ~\cite[(18.5)]{F}, we can replace the projective
dimension with that of the quasi-projective dimension. Note that by~\cite[(16.22)]{F} it is an extension of
\cite[(3.1)]{SY} to complexes.

\begin{prop}\label{T1}
Let $X\in{\mathcal D}_{b}^{f}(R)$ with $\Ci_R X<\infty$
and $Y\in{\mathcal D}_{+}^{f}(R)$. Then
\begin{itemize}
\item[(a)]
$\qpd_R X=\Ci_R X+\cx_R(X)$.
\vspace{0.05in}
\item[(b)]
$\dim_R Y\leq\dim_R(X\utp_R Y)+\qpd_R X.$
\end{itemize}
\end{prop}

\begin{proof}
\vspace{0.05in}

\noindent (a) From~\cite[(3.10)]{Sa} it follows that $\cx_R(X)<\infty.$ Thus $\Ci_R X+\cx_R(X)$ and $\qpd_R X$ are
finite. Let $P$ be a projective resolution of $X.$ We only need to show that
$\qpd_R P=\Ci_R P+\cx_R (P)$. Take $n\geq\sup X$, and let $C_n^ P$ be the $n$-th
cokernel of $P$. By~\cite[(2.12)]{Sa} and ~\cite[(3.7)]{Sa} we get $\cx_R(X)=\cx_R(C_n^P)$ and
$\Ci_R C_n^P<\infty,$ respectively. Hence by~\cite[(5.10)]{AGP} we can choose a
quasi-deformation $R\rightarrow R'\leftarrow Q$ of
codimension equal to $\cx_R(X)$ such that $\pd_Q C_n^{P'}<\infty$,
where $P'=P\otimes_R R'$. Let $P_{\leq n-1}$ and $P_{\geq n}$ be
the hard left and the hard right truncations of $P$, respectively.
The following sequence of complexes is
exact
$$\hspace{1.5in} 0\longrightarrow P_{\leq n-1}\longrightarrow P\longrightarrow
P_{\geq n}\longrightarrow0. \hspace{0.9in}(3.2.1)$$

\noindent Since $R'$ is faithful flat as an $R$-module, the
following sequence of $R'$-complexes is also exact
$$ 0\longrightarrow P'_{\leq n-1}\longrightarrow P'\longrightarrow
P'_{\geq n}\longrightarrow0.$$

\noindent It is clear that $\pd_{R'}P'_{\leq n-1}<\infty$. Since $\pd_Q R'<\infty$,
thus $\pd_Q P'_{\leq n-1}<\infty$. On the other hand, $P'_{\geq n}\simeq\Sigma^n C_n^{P'}$
and $\pd_Q C_n^{P'}<\infty$, thus we have $\pd_Q P'_{\geq n}<\infty$.
Now from the above exact sequence
we find that $\pd_Q P'<\infty$, and this implies that $\qpd_R P\leq\pd_Q
P'$. From~\cite[(3.3)]{Sa} we have the following (in)equalities

\vspace{0.01in}
$$\begin{array}{rl}
\qpd_R P-\Ci_R P &\,
\leq\pd_Q P'-\gd_R P \\
&\,=\gd_Q P'-\gd_R P \\
&\,=\gd_{R'}P'+\cx_R(P)-\gd_R P \\
&\,=\cx_R(P),
\end{array}$$

\noindent where the equalities hold by
~\cite[(2.3.10)]{Chr1}, ~\cite[(2.3.12)]{Chr1}, and ~\cite[(5.11)]{Chr2}, respectively.
On the other hand, there is a quasi-deformation of
codimension $c$, $R\rightarrow R''\leftarrow Q'$ such that
$\qpd_R P=\pd_{Q'} P''$, where $P''=P\otimes_{R}R''.$ Therefore from (3.2.1) we get $\pd_{Q'} C_n^{P''}<\infty$, and from
~\cite[(5.9)]{AGP} it follows that $\cx_R P=\cx_R(C_n^{P})\leq c$.
We have the following equalities

\vspace{0.01in}
$$\begin{array}{rl}
\qpd_R P-\Ci_R P &\,
=\pd_{Q'} P''-\gd_R P \\
&\,=\gd_{Q'} P''-\gd_R P \\
&\,=\gd_{R''}P''+c-\gd_R P \\
&\,=c\geq \cx_R(C_n^P)=\cx_R(P)
\end{array}$$

\noindent in which, we have used~\cite[(2.3.10)]{Chr1}, ~\cite[(2.3.12)]{Chr1}, and ~\cite[(5.11)]{Chr2} again.
Now assertion holds.

\vspace{0.1in}\noindent (b) Let $\rho:A\rightarrow B$ be a surjective homomorphism
of rings and let $X, Z\in{\mathcal D}_{b}^{f}(B)$. Then it is easy
to see that, $\dim_A(X\utp_A Z)=\dim_B(X\utp_B Z)$. Now let
$\qpd_R X=\pd_Q X'$, when $R\rightarrow R'\leftarrow Q$ is a
quasi deformation. From~\cite[(18.5)]{F} we get

$$\hspace{1.5in}\dim_Q Z'\leq\dim_Q(X'\utp_QZ')+\pd_Q X'. \hspace{0.9in} (3.2.2)$$

\noindent We have $\dim_Q Z'=\dim_{R'} Z'$ and
$\dim_Q(X'\utp_Q Z')=\dim_{R'}(X'\utp_{R'} Z')$, because $Q\rightarrow R'$ is surjective.
By the associativity of the derived tensor product
we get $X'\utp_{R'}Z'=(X\utp_R Z)'$. Since $R'$ is a flat $R$-algebra, from~\cite[(2.1)]{AF1}
we have the following equalities
$$\dim_{R'} Z'=\dim_R Z+\dim_R {R'/{\m}{R'}},$$

\noindent and
$$\dim_{R'}(X\utp_R Z)'=\dim_R(X\utp_R Z)+\dim_R {R'/{\m}{R'}}.$$

\noindent By combining (3.2.2) with the above equalities, we get the
following inequality

$$\dim Z\leq\qpd_R X+\dim_R(X\utp_R Z).$$

\noindent Now let $Y\in{\mathcal D}_{+}^{f}(R)$. Since $\H_{m}(Y)$ is a finite $R$-module, from the above inequality
it follows that
$$\hspace{1in}\dim_R \H_{m}(Y)\leq \qpd_R X+\dim_R(X\utp_R \H_m(Y)). \hspace{0.8in} (3.2.3)$$

\vspace{0.05in}\noindent Let $P$ be a projective resolution of $X$. Thus $P\simeq X$ and $P_i=0$, for $i\ll0$.
By applying ~\cite[(16.24.b)]{F} we get
$$\dim_R(P\otimes_R Y)=\sup\{\dim_R(P\otimes_R \H_m(Y))-m|m\in\mathbb{Z}\}.$$

\vspace{0.05in}\noindent Therefore
$$\dim_R(X\utp_R Y)=\sup\{\dim_R(X\utp_R \H_m(Y))-m|m\in\mathbb{Z}\}.$$

\noindent Now from (3.2.3) assertion holds.
\end{proof}

\begin{disc}
Let $\mathcal{F}(R)$ be the category of $R$-modules
of finite flat dimension. The \emph{large restricted flat dimension} of
$X\in{\mathcal D}_{+}(R)$, $\Rfd_R X$, was introduced and studied by Christensen, Foxby and Frankild in~\cite{CFF},
and is defined as the following
$$\Rfd_R X=\sup\{\sup(T\utp_R X)|T\in\mathcal{F}(R)\}.$$

\vspace{0.05in}\noindent It is shown that $\Rfd$ is a refinement of the flat dimension, \cite[(2.5)]{CFF}.
By~\cite[(3.6)]{SY1}, it is easy to see that $\Rfd_R X$ is also a refinement of $\Ci_R X$, for
$X\in{\mathcal D}_{b}^f(R).$ Using these facts, we can prove the first part of the above proposition, without any
using of the Gorenstein dimension. However, by~\cite[(12,13)]{F}, \cite[(2.6)]{I}, and \cite[(3.3)]{Sa},
we can also prove it, without any using of the above dimensions, see the proof of~\cite[(5.11)]{AGP}.
\end{disc}

\noindent The next proposition is the depth formula for the complete intersection
dimension of complexes, immediately follows from~\cite[(3.3)]{SSY} and~\cite[(3.3)]{Sa}.

\begin{prop}\label{T2}
Let $\Ci_R X<\infty$ and $Y\in{\mathcal D}_{b}^{f}(R)$, if
$\lambda=\sup(X\utp_RY)<\infty$, then
$$\depth_R(X\utp_R Y)=\depth_R Y-\Ci_R X.$$
\end{prop}

\vspace{0.05in}\noindent Now we are in the position of proving our main result.

\begin{thm}\label{T3}
Let $X, Y, Z\in{\mathcal D}_{b}^{f}(R)$, if
$\Ci_R X<\infty$ and $\sup(X\utp_R Y)< \infty$, then

\begin{itemize}
\item[(a)]
$-\sup\uhom_R(Z,Y)-\Ci_R X\leq-\sup\uhom_R(Z,(X\utp_RY)).$
\vspace{0.05in}
\item[(b)]

$-\sup\uhom_R(Z,(X\utp_R Y))\leq\amp Y+\cx_R(X)-\sup\uhom_R(Z, Y)-\inf X.$
\end{itemize}
\end{thm}

\begin{proof}
\vspace{0.05in}

\noindent (a) From~\ref{para3}(c) it follows that there is $\p\in\Supp_R X\cap\Supp_R Y\cap\Supp_R Z$ such that
$$-\sup\uhom_R(Z,(X\utp_R Y))=\depth_{R_{\p}}(X_{\p}\utp_{R_{\p}}Y_{\p})+\inf Z_{\p}.$$

\noindent Thus we have

\vspace{0.01in}
$$\begin{array}{rl}
 -\sup\uhom_R(Z,(X\utp_R Y)) &\,
=\depth_{R_{\p}}Y_{\p}-\Ci_{R_{\p}}X_{\p}+\inf Z_{\p} \\
&\,\geq\depth_{R_{\p}}Y_{\p}-\Ci_R X+\inf Z_{\p} \\
&\,\geq -\sup\uhom_R(Z,Y)-\Ci_R X
\end{array}$$

\noindent in which, the equality holds by Proposition~\ref{T2} and the inequalities hold by \cite[(3.4)]{Sa} and~\ref{para3}(c),
respectively

\vspace{0.1in}\noindent (b) By~\ref{para3}(c) for some $\p\in\Supp_R Z\cap\Supp Y$, we get
$-\sup\uhom_R(Z,Y)=\depth_{R_{\p}}Y_{\p}+\inf Z_{\p}$. Now let $W=(R/{\p}\utp_R X\utp_R Y)$. Since
$X\utp_{R} Y\in{\mathcal D}_{b}^{f}(R)$, thus we can choose a
prime ideal $\q$ of $R$ which is minimal in $\Supp W$. Therefore
$\p\subseteq\q$ and $\q\in\Supp X\cap\Supp Y$. On the one hand,
since $\p\in\Supp Z$, thus $\q\in\Supp Z$. Let
$V=(R/{\p}\utp_R Y)$. Since $V_{\q}\in{\mathcal D}_{+}^{f}(R_{\q})$ and $X_{\q}\in{\mathcal D}_{b}^{f}(R_{\q})$,
from Proposition~\ref{T1}(b) it follows that
$$\dim_{R_{\q}}V_{\q}\leq\qpd_{R_{\q}}X_{\q}+\dim_{R_{\q}}W_{\q}.$$

\noindent Using~\ref{para1} and the definition of dimension of complexes, it is straightforward to verify that
$\dim_{R_{\p}}V_{\p}=-\inf Y_{\p}$ and
$\dim_{R_{\q}}W_{\q}=-\inf X_{\q}-\inf Y_{\q}$.

By~\ref{para3}(a) and ~\ref{para3}(b), we have the following inequalities, respectively.

$$\dim R_{\q}/{\p}R_{\q}\geq\depth_{R_{\q}}Y_{\q}-\depth_{R_\p}Y_{\p}$$

\noindent and
$$\dim_{R_{\q}}V_{\q}\geq\dim_{R_\p}V_{\p}+\dim
R_{\q}/{\p}R_{\q}.$$

\noindent Therefore the following inequality holds

\[-\inf Y_{\p}+\depth_{R_{\q}}Y_{\q}-\depth_{R_\p}Y_{\p}\leq
\qpd_{R_{\q}}X_{\q}-\inf X_{\q}-\inf Y_{\q}.
\]
Proposition~\ref{T2} yields an equality
\[ \depth_{R_{\q}}Y_{\q}=\depth_{R_{\q}}(X_{\q}
\utp_{R_{\q}}Y_{\q})+\Ci_{R_{\q}}X_{\q}.\]

\noindent Now from Proposition~\ref{T1}(a) we get the following inequality
\[\depth_{R_\q}(X_\q\utp_{R_\q}Y_\q)
+\inf X_\q+\inf Y_\q-\inf Y_\p+\inf Z_\p\leq \cx_{R_\q}(X_\q)+\depth_{R_\p}Y_\p+\inf Z_\p. \]
\noindent It is easy to see that $\cx_{R_\q}(X_\q)\leq \cx_R(X)$,
and $\inf Z_\q\leq\inf Z_\p$. Thus from~\ref{para3}(c) we have
\[-\sup\uhom_{R}(Z, X\utp_{R} Y)+\inf X_{\q}+\inf Y_{\q}-\inf Y_\p\leq
\cx_R(X)-\sup\uhom_{R}(Z,Y) .\]

\noindent Therefore
\[-\sup\uhom_R(Z,X\utp_R Y)\leq\amp Y+\cx_R(X)-\sup\uhom_R(Z,Y)-\inf X.\]
Now assertion holds.
\end{proof}

\noindent Let $M$ and $N$ be $R$-modules. Then
$$\gr_R(M,N)=\inf\{i|\Ext_{R}^i(M,N)\neq 0\}.$$
\vspace{0.05in}\noindent If $\Ext_{R}^i(M,N)=0$, for all $i$, then $\gr_R(M,N)=\infty.$ We say that $M$ and $N$
are Tor-independent, if $\Tor_i^{R}(M,N)=0$, for $i>0$.

\vspace{0.05in}\noindent The following grade inequality, see \cite[(3.3)]{SY}, is an immediate corollary of the above theorem.

\begin{cor}\label{T4}
Let $M$, $N$ and $L$ be finite $R$-modules with $\Ci_R N<\infty$ such that $M$ and $N$ are
$\Tor$-independent $R$-modules. Then
$$\gr(L,M)-\Ci_R N\leq\gr(L,M\otimes_R N)\leq\gr(L,M)+\cx_R(N))$$
\end{cor}

\noindent In Example~\ref{ex1}, it is shown that $\gr(L,M\otimes_R N)$ may be arbitrary
greater than $\gr(L,M)-\Ci_R N$. This example also shows that in
the above inequality, the term $\cx_R(N)$ is strongly necessary.

\begin{para}\label{para4}
Let $k$ be a field and let $R_t$ for $t=1,2$
be $k$-algebras, and suppose that $M_t$ is an $R_t$-module. Let $R=R_1\otimes_k R_2$, $N=M_1\otimes_k R_2$,
and $M=M_2\otimes_k R_1$. It is easy to see that $\Tor_i^{R}(M,N)=0$, for $i>0$, see \cite[(4.2)]{J}, and
$N\simeq M_1\otimes_{R_1} R$, $M\simeq M_2\otimes_{R_2} R.$
\end{para}

\begin{ex}\label{ex1}
We use the above notations. Take an arbitrary integer $n\geq 1.$
Let $Q_1=k[[X_j,Y_j]], B_1=(X_j Y_j)$, for $1\leq j \leq n$ and $Q_2=k[[X_j,Y_j]], B_2=(X_j
Y_j)$, for $n< j \leq 2n$. It is clear that the local rings $R_t=Q_t/B_t$, for
$t=1,2$, are non-regular complete intersections of
codimension $n$. Let $M_1=R_1/(Y_i)R_1$
and $M_2=R_2/(Y_j)R_2$, for $1\leq i \leq n$ and $n< j \leq 2n$,
respectively. We have $\depth R_t=\depth_{R_t} {M_t}=n$, for $t=1,2$, and
the Auslander-Buchsbaum formula for the complete intersection
dimension yields $\Ci_{R_t}{M_t}=0$, for $t=1,2$.

\noindent Let $L_1=R_1/(X_i){R_1}$, for $1\leq i \leq n$. Theorem~\ref{T1}(b) and ~\cite[(16.22)]{F} yield
$\dim_{R_1} L_1\leq\cx_{R_1}(M_1)+\dim_{R_1}(L_1\otimes_{R_1} M_1).$
Since $L_1\otimes_{R_1} M_1\simeq k$ and $\dim_{R_1}
L_1=n$, thus $\cx_{R_1}(M_1)\geq n$, and from \cite[(8.1.2)]{Av1}
we have $n=\cx_{R_1}(k)\geq \cx_{R_1}(M_1)$, thus
$\cx_{R_1}(M_1)=n$. It is easy to see that
$R=k[[X_{\ell},Y_{\ell}]]/(X_{\ell}
Y_{\ell})$ and $M\otimes_R N=R/(Y_{\ell})R$, for
$1\leq\ell\leq 2n.$ By~\ref{para4}, $M$ and $N$ are $\Tor$-independent $R$-modules.
Now let $L=R/(X_{\ell})R$, for $1\leq\ell\leq 2n.$ Then we get $\gr_R(L,M)=n$, and
$\gr_R(L,M\otimes_R N)=2n.$ The natural local homomorphism $\varphi_t: R_t\rightarrow R$
is flat. Using~\ref{para4}, since $R$ is complete intersection, from \cite[(1.13.1)]{AGP} and
\cite[(5.2.3)]{AGP} we get $\Ci_R N=0$
and $\cx_R(N)=n,$ respectively. Therefor the left side of the inequality of Corollary~\ref{T4}
is equal to $n$, and it's right side is equal to $2n$.
\end{ex}


\begin{thebibliography}{10}

\bibitem{Av} L. L. Avramov, {\em Homological asymtotics of modules over local rings}, in "Commutative
Algebra," Vol. 15, pp. 33–62, MSRI, Berkeley, 1982; Springer-Verlag,
New York, 1989.

\bibitem{Av1} L. L. Avramov, {\em Infinite free resolutions}. Six
lectures on commutative algebra (Bellaterra, 1996) , 1--118, Progr.
Math. {\bf 166}, $Birkh\ddot{u}ser$, Basel, 1998.

\bibitem{AF} L. L. Avramov, H. B. Foxby, {\em Homological dimensions of unbounded complexes}, J. Pure
Appl. Algebra. {\bf 71} (1991), 129–155.

\bibitem{AF1} L. L. Avramov, H. B. Foxby, {\em Cohen-Macaulay properties of ring homomorphisms},
Adv. Math. {\bf 133} (1998), 54-95.


\bibitem{AGP} L. L.  Avramov, V. N. Gasharov, I. V. Peeva, {\em
Complete intersection dimension}, Inst. Hautes Etudes Sci. Publ.
Math. {\bf 86} (1997), 67--114.

\bibitem{Chr1} L. W. Christensen, {\em Gorenstein dimensions},
Lecture Notes in Mathematics, {\bf 1747}. Springer-Verlag, Berlin,
2000.

\bibitem{Chr2} L. W. Christensen, {\em Semi-dualizing complexes and
their Auslander categories}, Trans. Amer. Math. Soc. {\bf 353}
(2001), 1839--1883.

\bibitem{CFF} L. W. Christensen, H. B. Foxby, and A. Frankild, {\em Restricted homological dimensions
and Cohen-Macaulayness}, J. Algebra, {\bf 251} (2002), 479-502.

\bibitem{F} H. B. Foxby, {\em Hyperhomological algebra and
commutative rings}, K$\phi$benhavns Univ. Mat. Inst.
Preprint, 1998.

\bibitem{F1} H. B. Foxby, {\em Bounded complexes of flat modules}, J. Pure Appl. Algebra,
{\bf 15}, (1979), 149–172.

\bibitem{I} S. Iyengar, {\em Depth for complexes, and
intersection theorems}, Math. Z. {\bf 230} (1999), 545--567.


\bibitem{J} D. A. Jorgensen, {\em A generalization of the
Auslander-Buchsbaum formula}, J. Pure Appl. Algebra. {\bf 144}
(1999), 145--155.


\bibitem{SSY} P. Sahandi, T. Sharif, S. Yassemi, {\em Depth formula via complete
intersection flat dimension}, Comm. Algebra {\bf 39} (2011), 4002-4013.


\bibitem{Sa} S. Sather-Wagstaff, {\em Complete intersection
dimensions for complexes}, J. Pure Appl. Algebra. {\bf 190} (2004),
267--290.

\bibitem{SY1} T. Sharif, S. Yassemi, {\em Depth formula, restricted Tor-dimension under base change},
Rocky. Mountain. J. Math, {\bf 34} (2004), 1131-1146.

\bibitem{SY} T. Sharif, S. Yassemi, {\em Special homological
dimensions and intersection theorem}, Math. Scand. {\bf 96} (2005),
161--168.


\end{thebibliography}
\end{document}